\documentclass[a4,12pt]{article}
\usepackage{amssymb,latexsym,amsmath}
\usepackage{graphicx}
\usepackage{graphics}
\usepackage{caption}
\usepackage{subcaption}
\usepackage{tikz}
\usepackage[numbers,sort&compress]{natbib}
\newtheorem{theorem}{Theorem}
\numberwithin{theorem}{section}
\newtheorem{lemma}[theorem]{Lemma}
\newtheorem{proposition}[theorem]{Proposition}
\newtheorem{corollary}[theorem]{Corollary}
\newtheorem{definition}[theorem]{Definition}
\newtheorem{example}[theorem]{Example}
\newtheorem{remark}[theorem]{Remark}

\def\endproof{\ifSuppressEndOfProof\global\SuppressEndOfProoffalse
\else\xqed\fi\endfollowon}

\def\endfollowon{\endtrivlist}

\def\pushright#1{{\parfillskip=0pt\widowpenalty=10000
\displaywidowpenalty=10000\finalhyphendemerits=0\leavevmode\unskip \nobreak\hfil\penalty50\hskip.2em\null\hfill{#1}\par}}

\newif\ifSuppressEndOfProof\SuppressEndOfProoffalse
\def\xqed{\pushright\markendofproof}
\def\markendofproof{\rule{1.4ex}{1.4ex}}

\begin{document}
\title{ Common properties of some function rings on a topological space }
\author{M.R. Ahmadi Zand\thanks{Department of pure Mathematics, School of Mathematics , Yazd University, Yazd, Iran\newline{E-mail address:
\it mahmadi@yazduni.ac.ir}}}  
\maketitle
\begin{abstract}
 For a nonempty topological space X, the ring of all real-valued functions on $X$ with pointwise addition and multiplication  is denoted by $F(X)$ and continuous members of $F(X)$ is denoted by $C(X)$. Let $A(X)$ be a subring of $F(X)$ and $B$ be a non-zero and nonempty subset of $A(X)$. Then we show that there are a subset $S$ of $X$ and a ring homomorphism $\phi:A(X)\to A(S)$ such that $ker \phi =Ann(B)$. 
 A lattice ordered subring $A(X)$ of $F(X)$    is called $P$-convex if every prime ideal of $A(X)$ is  an absolutely convex ideal in $A(X)$. Some properties of $P$-convex subrings of $F(X)$ are investigated. We show that the ring of Baire one functions on $X$ is  $P$-convex. A  proper ideal $I$ in $A(X)$ is called a  pseudofixed ideal if  $\bigcap \overline{Z[I]}\neq \emptyset $, where $ \overline{Z[I]}=\{cl_X f^{-1}(0) | f\in I\}$. Some characterizations of pseudofixed ideals in some subrings of $F(X)$ are given. Let $X$ be a completely regular Hausdorff space and let  $A(X)$ be a subring of $F(X)$ such that 
 $f \in A(X) $  is a unit of $A(X) $ if and only if  $ f^{-1}(0)=\emptyset$ and $C(X) \subseteq F(X)$. Then we show that $A(X)$ is a Gelfand ring and 
 $X$ is compact if and only if every proper ideal of $A(X)$ is pseudofixed.
  \end{abstract}
\vspace{0.5 cm}
\section{Introduction}
Let $X$  be   nonempty topological  space,  the set of all functions from $X$ to $\mathbb{R}$ is denoted by $F(X)$, the  continuous members of $F(X)$ is denoted by $C(X)$ and  collection of all pointwise limit functions of sequnces in $C(X)$ called the ring  of Baire class one functions or Baire one functions  is denoted by $B_{1}(X)$ \citep{Deb Ray}. 
  The ring
of all constant function on $X$  is denoted by  $\mathbb{R}(X)$  and if $r\in \mathbb{R}$, then ${\bf r}$ denotes a member of  $\mathbb{R}(X)$ defined by ${\bf r}(x)=r$ for any $x\in X$. Let $S$ be a subset of $X$.
The characteristic function of $S$ is denoted by $\chi_S$.
   As usual $cl_X A=\overline{A}$ and $int_X A=A^\circ$ will denote the closure and interior of a subset $A$ of $X$, respectively. 
   Let  $f\in F(X)$, then $Z(f)$ ($C(f)$) denotes the set $\{x\in X |f(x)=0\}$ ($\{x\in X:~f$ is continuous at $x \}$)  and  $Z(f)$ is  called a zero-set.  The ring
of all $f\in F(X)$   such that $X\setminus C(f)$ is a finite set is denoted by $C(X)_F $ \cite{GGT}. Clearly, $ C(X)_F $ is a subring of $F(X)$ and it is an over-ring of $C(X)$ for any topological space $X$ \cite{ZandKHos}. 

 We note that $F(X)=C(Y)$, where the topological space  $Y$ is the set $X$ with the  discrete topology. Thus by  \cite{Gillman},  the following equalities  hold for all $n,m\in \mathbb{N}$ and $a,b \in  F(X)$.
 \begin{itemize}
\item[(i)] $Z(b)=Z(|b|)=Z(b^n)$;
\item[(ii)] $Z(a^{m}b^{n})=Z(a)\cup Z(b)$;
\item[(iii)] $Z(n)=\emptyset$ and $Z(0)=X$;
\item[(iv)]  $Z(|a|^{m}+|b|^{n})=Z(a^{2n}+b^{2m})=Z(a)\cap Z(b)$. 
\end{itemize}  

Obviously, the above properties are true in any subring of $F(X)$ with different notation  (for  example see   \cite{Ahmadi Zand, AHkh,ZandKHos, ZandKHos2,Deb Ray,  GGT,Gillman}).  All rings are assumed to be commutative and reduced. A non-zero ideal $E$
in a commutative ring $R$ is called essential if it intersects every non-zero ideal non- trivially. For any $r\in R$, the intersection of all maximal
ideals  in $R$ containing $r$ is denoted by $M_r$,  and an ideal $I$ in $R$ is called a $z$-ideal if  $M_r \subseteq I$
for all  $r\in I$, see \cite[4A]{Gillman}. We note that a  $z$-ideal in $B_1(X)$ is denoted  by $Z_B$-ideal in \cite{Deb Ray, Deb Ray2}. For each subset $S$ of  $R$, the annihilator of $S$ is denoted by $Ann(S)$.\\
In this paper, we consider three properties of C(X) and with respect to each of these properties subrings of $F(X)$ are classified and studied. We also provide examples to illustrate the results
presented herein.  For undefined notations, the reader is referred to \cite{Engelking} and
\cite{Gillman}. 

\subsection{Definition and preliminaries}
The following theorem which is a generalization of \cite[Theorem 2.2]{Deb Ray} is well known but for the sake of completeness, we give a proof here.
\begin{theorem}\label{1.1}
Let $X$ be any topological space and $f\in B_1(X)$. Then, $f$ is a unit of $B_1(X)$ if and only if $Z(f)= \emptyset $.
\end{theorem}
\begin{proof}
If $f$ is a unit of $B_1(X)$, then it is straightforward that $Z(f)=\emptyset$. Conversely, if $Z(f)=\emptyset$, then $f^2 >0$ and so by \cite[Theorem 2.3]{Deb Ray} there exists $g\in B_1(X)$ such that $gf^2=1$. Thus $f$ is a unit of $B_1(X)$.
\end{proof}
 Let $X$ be a topological space and $A(X)$ be a subring of $F(X)$ determined by some special property or properties, for example $A(X)$ can be equal to $\mathbb{R}(X)$, ~$C(X)$, $C(X)_F$ or $B_1(X)$. Bounded members of $A(X)$  is denoted by $A^*(X)$, clearly $A^*(X)$ is a subring of $A(X)$. If $f\in F(X)$ and $S\subseteq X$, then  $f_{|_{S}}\in F(S)$ denotes the restriction of $f$ to $S$.\\
   Recall that if $ \mathcal{P} $ is a collection of subsets of a topological space $ X $ which is closed under finite    unions and intersections,  then a  $\mathcal{P}$-filter on X is a collection $\mathcal{F}$ of  subsets of $X$ with the following properties. \begin{center} $\emptyset \notin \mathcal{F}$, if $P_1, P_2 \in \mathcal{F}$, then 
$ P_1 \cap P_2\in \mathcal{F} $ and  $ P_1 \in \mathcal{F}, P_1 \subset P_2 \in \mathcal{P}  $ implies  $P_2\in \mathcal{F}$. \end{center}
If $A(X)$ is a subring of $F(X)$  then 
the
family $\{Z(f): f\in A(X) \}$    denoted by
$Z_{A(X)}$ is closed under finite unions and intersections. In \cite{Gillman}, $ Z_{C(X)}=Z_{C^*(X)} $ is denoted by $Z(X)$,  in \cite{Deb Ray}, $ Z_{B_1(X)} $  is denoted by $ Z(B_1(X)) $ and in \cite{GGT}, $ Z_{C(X)_F}$ is  denoted by $\mathcal{Z}[C(X)_F] $ or $\mathcal{Z}(X)$.  We note that $Z_{C(X)}$-filters and $z$-filters \cite{Gillman} are the same. In this paper, a $\mathcal{Z}$-filter \cite{GGT} is called a $Z_{C(X)_F}$-filter and a $Z_B$-filter \citep{Deb Ray2} is called a $ Z_{B_1(X)} $-filter.\\
 Generally, if $A(X)$ is a subring of $F(X)$ we can consider $Z_{A(X)}$-filters and if $\mathcal{F}$ is a $Z_{A(X)}$-filter, then  we denote $\{ f \in  A : Z(f) \in \mathcal{F} \}$ by $Z^{-1}[\mathcal{F}]$ and if $I$ is an ideal in  $A(X)$, then we denote 
$\{Z(f) : f \in  I \}$  by $Z[I]$.
\section{Fixed ideals in  subrings of $F(X)$}
 Recently, a fixed  (resp., free) ideal was defined and studied in some especial over-rings of $C(X)$ for example see \cite{Deb Ray2}, \cite{GGT}. Similar to these definitions we can define a fixed  (resp., free) ideal in an arbitrary subring of $F(X)$.
 \begin{definition}
Let $A (X)$ be a subring of $F(X)$. An ideal $I$  in $A (X)$ is called a {\it fixed} ideal  if $\bigcap Z[I] $ is nonempty and  $I$ is called a {\it free} ideal if $\bigcap Z[I]=\emptyset$.   
\end{definition}
Clearly, every free ideal in  a subring of $F(X)$ is not contained in any fixed ideal.
\begin{proposition}\label{compact}
Let $I$ be a proper free ideal in a subring $A(X)$  of $F(X)$, $\mathbb{R}(X) \subseteq A(X)$  and $p\in X$. Then there exists a non-zero ideal $J$ contained in $I$ such that $p\in \bigcap Z[J]$ or $X$ is a disjoint union of two proper zero-sets in $Z_{A(X)}$. 
\end{proposition}
\begin{proof}
Since $I$ is free, there is $f\in I$ such that $r=f(p)\neq 0$ and  so $g=f(f-{\bf r}) \in I$. The ideal generated by $g$ in $A(X)$ is denoted by $J$.  If $g\neq 0$, then  $J$ is a nonzero ideal containted in $I$ and $p\in \bigcap Z[J]$. Otherewise, $ Z(f-{\bf r})$ is a proper subset of $X$ since $I$ is a proper ideal. It is clear that $X$ is a disjoint union of  $Z(f)$ and $ Z(f-{\bf r})$ which completes the proof. 
 \end{proof}

  \begin{theorem}
 Let $A(X)$ be a subring  of $F(X)$ and $ B$ be a non-zero and nonempty subset of $A(X)$. Then there are a subset $S$ of $X$ and a ring homomorphism $\phi:A(X)\to A(S)$ such that $ker \phi =Ann(B)$ is a fixed ideal. 
 \end{theorem}
   \begin{proof}
If $S=\bigcup_{b\in B}(X\setminus Z(b))$, then a function $\phi:A(X)\to A(S)$ defined by $\phi(f)=f_{|_S}$ is a ring homomorphism. Clearly, $ker \phi$ is fixed and $f\in Ann(B)$ implies that  for any $b\in B$, $bf=0$ or $X\setminus Z(b)\subseteq Z(f)$. Thus $f_{|_S}=0$, i.e., $f\in ker \phi$, and so $ Ann(B) \subseteq ker \phi$.  Conversely, if $f\in ker \phi$, i.e.,  $f_{|_S}=0$, then  for any $b\in B$,  we have $f_{|_{X\setminus Z(b)}}=0$ so  $fb=0$, i.e., $f\in Ann(B)$. Therefore, $Ker \phi \subseteq Ann(B)$ which completes the proof.
 \end{proof}
 \begin{theorem}
 Let $A,B, S$ be  nonempty subsets of a completely regular Hausdorff space $X$, $A(X)$ be a subring of $F(X)$, $C(X)\subseteq A(X)$ and   $L_S=\{f\in A(X)| \overline{S}\subseteq Z(f)\}$. Then the following statements hold.
 \begin{itemize}
 \item[{\normalfont(1)}]
  $L_S$ is a fixed ideal in $A(X)$.
 \item[{\normalfont(2)}]
 $\overline{A}=\overline{B}$ if and only if $L_A=L_B$.
 \item[{\normalfont(3)}]
 $L_B=\{0\}$ if and only if $B$ is a dense subset of $X$.
 \end{itemize}
 \end{theorem}
 \begin{proof} (1) It is straightforward. \\
 (2) Let  $\overline{A}=\overline{B}$. Then $L_A=L_{\overline{A}}=L_{\overline{B}}=L_B$. Conversely, if $\overline{A}\neq \overline{B}$, then without loss of generality we can assume that there is $p$ in $\overline{A} \setminus \overline{B}$. By  complete regularity of $X$, there is     $f_{p}\in C(X)\subseteq A(X)$ such that $p\notin Z(f)$ and $\overline{B}\subseteq Z(f_{p})$.
 Thus,         $f_{p}\in L_B\setminus  L_A$.       \\
 (3)  If $B$ is not dense in $X$, then there exists $p\in X\setminus \overline{B}$. As we have shown above, this means that  there exists a non-zero   $   f_{p}$ in   $ L_B$, i.e., $L_B\neq \{0\}$. The converse is obvious.                                                                          
 \end{proof}
\begin{lemma}\label{27kh}
Let $X$ be a topological space, $A(X)$ be a subring of $F(X)$ and $\mathbb{R}(X) \subseteq A(X)$. Then for any $p\in X$, $M_p^{A(X)}=\{f\in A(X)| f(p)=0\}$ is a fixed maximal ideal in $A(X)$.
\end{lemma}
\begin{proof}
For any $p\in X$, a function $\phi: A(X) \to \mathbb{R}$ defined by $\phi(f)=f(p)$ is a ring homomorphism and since $A(X)$ contains all  of the constant functions, $\phi$ is onto. Clearly, $ker (\phi)=M_p^{A(X)}$ and so it is a maximal ideal  in $A(X)$.
\end{proof}
\begin{corollary}\label{khorob10tir}
Let $X$ be a topological space, $A(X)$ be a subring of $F(X)$ and $\mathbb{R}(X) \subseteq A(X)$.  Then  the intersection of all maximal ideal in $A(X)$ is zero, i.e., $A(X)$ is semi-simple.
\end{corollary}

 \begin{proposition}\label{5aya} 
 Let $X$ be a completely regular Hausdorff space. If $A(X)$ and $B(X)$ are two subrings of $F(X)$ and two over-rings of $C(X)$, then there is a one-to-one correspondence between collections of all fixed maximal ideals in $A(X)$ and $B(X)$.
\end{proposition}
\begin{proof}
For any $p\in X$,  $M_p^{A(X)}$ is a fixed maximal ideal in $A(X)$ by Lemma \ref{27kh}.  $M_p^{A(X)}\cap C(X)=M^p$ is a fixed maximal ideal in $C(X)$. If $p,q$ are two distinct points of $X$, then from $M^p \neq M^q$ \cite{Gillman} it follows that $M_p^{A(X)}\neq M_q^{A(X)}$ so there is a bijection between the collections of fixed maximal ideals in $A(X)$ and $C(X)$.  Similarly, there is a bijection  between  the collections of all fixed maximal ideals in $B(X)$ and $C(X)$ which completes the proof.
\end{proof}
\begin{proposition}
Let $X$ be a completely regular Hausdorff space, $A(X)$ be a subring of $F(X)$ and $C(X)\subseteq A(X)$. Then every  fixed maximal ideal in $A(X)$ is an essential ideal if and only if for any $p\in X$, $\chi_{\{p\}} \notin A(X)$.
\end{proposition}
\begin{proof}
If  there is  $p\in X$ such that  $e=\chi_{\{p\}} \in A(X)$, then $M_p^{A(X)}$ is the ideal generated by $1-e$. Thus $M_p^{A(X)}$ is not an essential ideal. Conversely, let  for any $p\in X$, $\chi_{\{p\}} \notin A(X)$ . If $M$ is a  fixed maximal ideal in $A(X)$, then there exists $q\in X$ such that $M=M_q^{A(X)}$ by Lemma \ref{27kh}.  The equality $\bigcap Z[M]=\{q\}$ holds by proof of Proposition \ref{5aya}. Thus, if $f$ is a non-zero element of $A(X)$, then $X\setminus Z(f)\neq \{q\}$ and so $M$ intersects the ideal generated by $f$   in $A(X)$ which completes the proof.
\end{proof}

\section{When a unit   is equivalent to its zero-set is empty}
In this section, we study  a subring $A(X)$ of $F(X)$ such  that $f\in A(X)$ is a unit if and only if  $Z(f)=\emptyset$. A generalization of fixed ideals in an arbitrary subrings of  $F(X)$ are given and some characterizations of this kind of ideals  in some subrings of $F(X)$ are established. Also, we give a new algebraic characterization of compact spaces. \\
 Some subrings of $F(X)$ have the following property  (\ref{mahmood}) of  $C(X)$.
 Let a  subring $ A (X)$ of $ F(X) $  have the following property. 
\begin{equation}\label{mahmood}
f \in A (X) ~\textit{ is a unit of }~ A (X) ~\textit{ if and only if } ~ Z(f)=\emptyset.
\end{equation} 
\begin{remark}\label{kalb}
Let  $X$ be  a topological space.  Then by  {\normalfont Theorem  \ref{1.1}}, the condition  {\normalfont(\ref{mahmood})} holds for $A(X)=B_1(X)$. Also,  the condition  {\normalfont (\ref{mahmood})} holds for $A(X)= C(X)_F $   {\normalfont \cite[Lemma 2.4]{GGT}}. If $X$ is not  pseudocompact, then   the condition {\normalfont  (\ref{mahmood})} is not true for $A^*(X)$, where $C(X) \subseteq A(X)$ since it is not true for $C^*(X)$.
\end{remark}
The proof of the following result is similar to  \cite[Theorems 2.3 and 2.6]{Gillman} and \cite[Theorems 2.5, 2.6,2.8 and 2.9]{Deb Ray2} and so we state it without proof.
  \begin{proposition}\label{KHorshidKhordad}
Let the condition  {\normalfont (\ref{mahmood})} be true for a subring $A(X)$ of $F(X)$. Then the following statements hold.
\begin{itemize}
\item[{\normalfont(1)}]
for any proper ideal $I$ in   $A (X)$, $Z[I]$ is a $Z_{A (X)}$-filter and $ Z^{-1}[Z[I]]\supseteq I $.
\item[{\normalfont(2)}]
if $\mathcal{F}$ is a $Z_{A (X)}$-filter  on X, then $Z^{-1}[\mathcal{F}]$ is a proper $z$- ideal in $A(X)$.
\item[{\normalfont(3)}] 
Let $M$ be a maximal ideal in $A(X)$ and $Z(f)$ meet every element of $Z[M]$ for some $f\in A(X)$, then $f\in M$.
\end{itemize}
 \end{proposition}

If $I$ is a proper free ideal in $A(X)$ such that the condition {\normalfont (\ref{mahmood})} holds for $A(X)$ and $f\in I$, then $f$ is not a unit of $A(X)$ and so $Z(f)\neq \emptyset $. Thus the  fixed ideal generated by $f$   is contained in $I$.  Since $Z[Z^{-1}[\mathcal{F}]]=\mathcal{F}$, every $Z_{A (X)}$-filter $ \mathcal{F} $ is of the form $Z[I]$ for some proper ideal $I$ in
$A (X)$.   In addition, let  $\mathbb{R}(X) \subseteq A(X)$ and $f\in A(X)$. Then a topological  characterization of the ideal  $M_f$ is given in the following result.
\begin{proposition}\label{tankin}
Let $A(X)$ be a subring of $F(X)$, the condition  {\normalfont (\ref{mahmood})} hold for it,  $\mathbb{R}(X) \subseteq A(X)$ and $f\in A(X)$. Then  $M_f=\{g\in A(X) | Z(f) \subseteq Z(g) \}$.
\end{proposition}
\begin{proof}
If $Z(g)$ does not contain $Z(f)$ for some $g\in A(X)$, then there is $p\in Z(f) \setminus Z(g)$. By Lemma \ref{27kh},    $M_p^{A(X)}$ is a fixed maximal ideal in $A(X)$ which contains $f$ but does not contain $g$, i.e.,  $g\notin M_f$. Conversely, let $Z(f)\subseteq Z(g)$ for some $g\in A(X)$. If $M$ is a maximal ideal in $A(X)$  containing $f$, then $Z(g) \in Z[M]$ and $Z[M]$ is a $Z_{A(X)}$-filter  by Proposition \ref{KHorshidKhordad}. Therefore, $ Z^{-1}[Z[M]]$ containing $g$ is a proper ideal in $A(X)$ by Proposition \ref{KHorshidKhordad} and so $g\in  Z^{-1}[Z[M]]=M$ which completes the proof.
\end{proof}

A topological characterization of a $z$-ideal in some subrings of $F(X)$ is given in the following result.
\begin{proposition}\label{zideal}
Let $A(X)$ be a subring of $F(X)$ such   that   the condition {\normalfont(\ref{mahmood})} holds for  it and $\mathbb{R} \subseteq A(X)$. Then for a proper ideal $I$ in $A(X)$,  the following conditions  are equivalent.
\begin{itemize}
 \item[{\normalfont (1)}]
$I$  is a $z$-ideal.
 \item[{\normalfont (2)}]
whenever $Z(f)\subseteq Z(g)$, $f\in I$ and $g\in A(X)$, then $f\in I$.
 \item[{\normalfont(3)}]
$ Z^{-1}[Z[I]]=I$.
\end{itemize}
\end{proposition}
\begin{proof}
Clearly, (2) and (3) are equivalent.  Similar to the proof of  Proposition  \ref{tankin} and \cite[4A]{Gillman} (1) and (2) are equivalent.
\end{proof}
The following two theorems can be proved similar to  \cite[Theorems 2.16 and 2.17]{Deb Ray2} and so we state them without proof.
\begin{theorem}\label{3deyvan}
Let $A(X)$ be a subring of $F(X)$ such   that   the condition {\normalfont (\ref{mahmood})} holds for $A(X)$. For any proper $z$-ideal $I$ in $A(X)$, the following conditions are equivalent.
\begin{itemize}
\item[{\normalfont (1)}]
$I$ is a prime ideal in $A(X)$.
\item[{\normalfont (2)}]
A prime ideal in $A(X)$ is contained in $I$.
\item[{\normalfont (3)}]
From $f,g\in A(X)$ and $fg=0$ it follows that $f\in I$ or $g\in I$.
\item[{\normalfont (4)}]
For all $f\in A(X)$, there is a zero-set of $I$ on which $f$ does not change sign.
\end{itemize}
\end{theorem}
\begin{theorem}\label{etir}
Let $A(X)$ be a subring of $F(X)$ such   that   the condition {\normalfont  (\ref{mahmood})} holds for $A(X)$. Then, every prime ideal in $A(X)$ can be extended to a unique maximal ideal in $A(X)$, i.e.,   $A(X)$ is a    Gelfand ring.
\end{theorem}
It is well known that if the intersection of all maximal ideals in a comutative   Gelfand ring is zero, then every prime ideal is either an
essential ideal or a maximal ideal generated by an idempotent \cite{samei}.  Thus, we have the following result.
\begin{corollary}
Let $A(X)$ be a subring of $F(X)$ such   that   the condition {\normalfont (\ref{mahmood})} holds for $A(X)$ and $\mathbb{R}(X) \subseteq A(X)$. Then every prime ideal in $A(X)$ is either  an essential  ideal or a  maximal fixed ideal which is at the same time a minimal prime generated by an idempotent.
\end{corollary}
\begin{proof}
By Corollary \ref{khorob10tir}, Theorem \ref{etir} and the above comment, it is trivial.
\end{proof}

\begin{definition}
Let $A(X)$ be a subring of $F(X)$, $\mathcal{F}$ be a $Z_{A(X)}$-filter and $\overline{\mathcal{F}}=\{ \overline{F}| F\in \mathcal{F} \}.$ Then  a proper ideal $I$ in $A(X)$ is called a  pseudofixed ideal if  $\bigcap \overline{Z[I]}\neq \emptyset .$
\end{definition}
We note that if $\mathcal{F}$ is a $Z_{C(X)}$-filter or $Z_{C^*(X)}$-filter, then $\overline{\mathcal{F}}=\mathcal{F}$ and so the concepts pseudofixed ideals and fixed ideals are coincided in $C(X)$,  $C^*(X)$ or their subrings. \\
It is well known that a $T_1$-space $X$ is finite if and only if every proper ideal in $C(X)_F$ is  a fixed  ideal in $C(X)_F$ \cite[Theorem 3.2]{GGT}, but the following example shows that this is not true for a pseudofixed ideal in $C(X)_F$.
\begin{example}
Let $S=\{\chi_{(0, \frac{1}{n}) }| n\in \mathbb{N} \}$ and $I$ be the ideal generated by $S$ in $C(\mathbb{R})_F$. It is easily seen that $I$ is a pseudofixed  ideal  in $C(\mathbb{R})_F$ which is not a  fixed ideal in $C(\mathbb{R})_F$.
\end{example}
 Now we answer the question of when every proper ideal in a subring of $F(X)$    is pseudofixed. 
\begin{proposition}\label{9tir}
Let $X$ be a topological space and $A(X)$ be a subring of $F(X)$ such   that   the condition {\normalfont(\ref{mahmood})} is true for $A(X)$. Then the following statements hold.
\begin{itemize}
\item[{\normalfont(1)}]
If $X$ is compact, then every proper  ideal in $A(X)$ is pseudofixed.
\item[{\normalfont(2)}]
Let $X$ be a completely regular Hausdorff space and $C(X)\subseteq A(X)$. If every proper  ideal in $A(X)$ is pseudofixed, then $X$ is compact. 
\end{itemize}
\end{proposition}
\begin{proof}
(1) Let $X$ be compact and $I$ be a proper ideal in $A(X)$. Then  $Z[I]$  is a $Z_{A(X)}$-filter by Proposition \ref{KHorshidKhordad} and so it has  the finite intersection property. Therefore $\overline{Z[I]}$ has the finite intersection property   and so by the compctness of $X$, $\bigcap \overline{Z[I]}\neq \emptyset$, i.e., $I$ is a pseudofixed ideal in $A(X)$. \\
(2) Let  every proper ideal in $A(X)$ be pseudofixed, $J$ be a proper ideal in $C(X)$ and $I$ be the ideal in $A(X)$ generated by $J$. Then $I$ is a proper ideal in $A(X)$ by the condition (\ref{mahmood}),  and so $I$ is  a pseudofixed ideal in $A(X)$. Thus, 
$$\emptyset \neq \bigcap \overline{Z[I]} \subseteq \bigcap \overline{Z[J]}=\bigcap Z[J],$$ which implies that $J$ is a fixed ideal in $C(X)$. Therefore by \cite[Theorem 4.11]{Gillman}, $X$ is compact.
\end{proof}
It is well known that a completely regular Hausdorf space $X$ is compact if and only if every proper ideal in $C(X)$ is fixed \cite{Gillman}. In the following proposition, we give some new characterizations  of this result.
\begin{proposition}
For a completely regular  Hausdorff space $X$ the following are equivalent.
\begin{itemize}
\item[{\normalfont(1)}]
Let  the condition {\normalfont (\ref{mahmood})} is true for a subring $A(X)$ of $F(X)$ and $C(X)\subseteq A(X)$ . Then every proper ideal in  $A(X)$   is  pseudofixed.
\item[{\normalfont(2)}]
There is a subring $A(X)$ of $F(X)$ such  that $C(X)$ is a subring of it,   the condition  {\normalfont (\ref{mahmood})} holds for $A(X)$ and every  proper ideal in  $A(X)$ is pseudofixed.
\item[{\normalfont(3)}]
$X$ is compact.
\item[{\normalfont(4)}]
Every proper ideal in $C(X)$ is fixed. 
\end{itemize} 
\end{proposition}
\begin{proof}
By Proposition \ref{9tir} and the above comment, it is straightforward.
\end{proof}
For a Hausdorff space $X$ and subrings of $F(X)$ that the condition (\ref{mahmood}) is true for them, a topological charactrization of free ideals which are not pseudofixed are given in the following theorem.
\begin{theorem}\label{mozafar}
Let $X$ be a Hausdorff space and $A(X)$ be a subring of $F(X)$ such   that   the condition {\normalfont (\ref{mahmood})} holds for it. A proper ideal $I$ in $A(X)$ is not pseudofixed if and only if for any compact subset $Y$ of $X$ there exists $f\in I$ such that $\overline{Z(f)} \bigcap Y=\emptyset$.
\end{theorem}
\begin{proof}
If $I$ is not pseudofixed, then for every $x\in X$ there is $f_x \in I$ such that $x\notin \overline{Z(f)}$. Given a compact subset $Y$ of $X$, by the compactness of $Y$ there exists a finite subset $F$ of $Y$ such that
 \begin{equation*}
Y\subseteq \bigcup_{y\in F}(X\setminus \overline{Z(f_y)}.
\end{equation*}
Thus by the properties of zero-sets,
\begin{equation*}
Y\subseteq X \setminus \bigcap_{y\in F} \overline{Z(f_y)}\subseteq X\setminus \overline{\bigcap_{y\in F} Z(f_y)}=X\setminus \overline{Z(\sum_{y\in F}f^2_y)}.
\end{equation*}
Therefore, $f=\sum_{y\in F}f^2_y \in I$ and $Y\bigcap \overline{Z(f)}=\emptyset$. Conversely, if for every compact subset $Y$ of $X$ there is $f_Y \in I$ such that $Y\bigcap \overline{Z(f_Y)}=\emptyset$, then for any $x\in X$ there is $f_{\{x\}}\in I$ such that $\{x\}\bigcap \overline{Z(f_{\{x\}})}=\emptyset$ and so $I$ is not  pseudofixed.
\end{proof}
A topological characterization of  pseudofixed ideals  in some subrings of $F(X)$ is given in the following result.
\begin{corollary}
Let $X$ be a Hausdorff space and $I$ be a proper ideal in an arbitrary subring $A(X)$ of $F(X)$ such   that   the condition {\normalfont (\ref{mahmood})} holds for $A(X)$. Then $I$ is a  pseudofixed ideal in $A(X)$ if and only if there exists a compact subset $Y$ of $X$ such that $Y$ intersects the closure of every zero-set in $Z[I]$.
\end{corollary}
\begin{corollary}
Let $X$ be a Hausdorff space, $A(X)$ be a subring of $F(X)$ such   that   the condition {\normalfont (\ref{mahmood})} holds for $A(X)$ and $Y$ be a compact subset of $X$. If $I$ is a proper ideal in $A(X)$ that is not pseudofixed, then the image of a function $\phi: I \to A(Y)$ defined by $\phi(f)=f_{|_Y}$ is a subring of $A(Y)$.
\end{corollary}
\begin{proof}
If $J=\phi[I]$, then  $0\in J$ and if $f,g\in I$, then $f_{|_{Y}}g_{|_{Y}}=(fg)_{|_{Y}}\in J$ and $f_{|_{Y}}+g_{|_{Y}}=(f+g)_{|_{Y}} \in J$. By Theorem \ref{mozafar}, there is $f\in I$ such that  $\overline{Z(f)} \bigcap Y=\emptyset$. Thus, $Z(f_{|_{Y}})=\emptyset$ and the condition (\ref{mahmood}) implies that $f_{|_{Y}}$ is a unit of  $A(Y)$, i.e., ${\bf 1}\in J$ which completes the proof.
\end{proof}
\section{composition with a continuous function}
In this section,  we will prove that every prime ideal in $B_1(X)$ is  absolutely convex. we will study a subring $A(X)$ of $F(X)$ such that $g\in C(\mathbb{R})$ and $f\in A(X)$ imply $g\circ f \in A(X)$ and we will prove that $A(X)$ is lattice ordered ring. Also,  a contravariant functor of the category of all topological spaces and the continuous mappings into the category of commutative lattice ordered rings will be stablished. \\
Let $X$ be a topological space. Some subrings of $F(X)$ have the following property of rings of continuous real-valued functions on $X$.\\
 A subring $A(X)$ of $F(X)$ has the following property.
\begin{equation}\label{29ordybehesht}
\text{If}~ g:\mathbb{R}\to \mathbb{R} ~ \text{ is a  continuous function, then for any} ~ f\in A(X), \text{ we have}~ g\circ f \in A(X).
  \end{equation}
\begin{example}\label{Ex}
Let $X$ be an arbitrary topological space and $g\in C(\mathbb{R})$. If  $f\in C(X)_F$, then $C(f)\subseteq C(g\circ f)$ and so $g\circ f \in C(X)_F$. If $f\in B_1(X)$, then  $ g\circ f \in B_1(X)$ by {\normalfont \cite[Theorem 2.4]{Deb Ray}}. Thus the condition {\normalfont (\ref{29ordybehesht})} holds for subrings $C(X)_F$ and $B_1(X)$ of $F(X)$.
\end{example}
\begin{proposition}
Let  the condition {\normalfont (\ref{29ordybehesht})} hold for a subring $A(X)$ of $F(X)$. If $F$ is a $\mathbb{R}$ and $f\in A(X)$, then $f^{-1}(F) \in Z_{A(X)}$.
\end{proposition}  
\begin{proof}
There exists $g \in C(\mathbb{R})$ such that $Z(g)=F$. Also, $g\circ f \in A(X)$   by hypothesis and  $Z(g\circ f)=f^{-1}(F)$ which completes the proof.
\end{proof}
\begin{remark}
Let $\mathbb{R}[x]$ denote the  polynomial ring with coefficients in $\mathbb{R}$. Then $\mathbb{R}[x]$ is a subring of $F(\mathbb{R})$ and we can consider $g(x)=\frac{1}{|x|^\frac{1}{3}+1}\in C(\mathbb{R})$ and  $f(x)=x+1 \in \mathbb{R}[x]$
which imply that $g\circ f \notin \mathbb{R}[x]$. Thus, the condition {\normalfont (\ref{29ordybehesht})} does not hold for the subring $\mathbb{R}[x]$.
\end{remark}
Let $S$ be a nonempty subset of $F(X)$ and $n\in \mathbb{N}$ be an odd number. For any $x\in X$, set 
$(f^\frac{1}{n})(x)=f(x)^\frac{1}{n}$. Thus $\{f^\frac{1}{n}|f\in S\}$ denoed by $S^\frac{1}{n}$ is a subset of $F(X)$ and this notation is used in the following proposition.                 
\begin{proposition}
Let  the condition  {\normalfont(\ref{29ordybehesht})} hold for a subring $A(X)$  of $F(X)$. If $P_*$ and $Q_*$  are two proper ideals in $A(X)$ such that $P^\frac{1}{3}_*Q^\frac{1}{5}_*$ is a prime ideal, then $P_*$ or $Q_*$ is prime in $A(X)$.
\end{proposition}   
\begin{proof}
Since the condition {\normalfont (\ref{29ordybehesht})} hold for  $A(X)$,  $Q=Q^\frac{1}{5}_* ,~ P=P^\frac{1}{3}_*$ are two ideals in $A(X)$. If $fg\in P\bigcap Q$, then $f^2g^2 \in PQ$. Since $PQ$ is prime, $f^2\in PQ$ or $g^2 \in PQ$ and so $f\in PQ\subseteq P\bigcap Q$ or $g\in PQ\subseteq P\bigcap Q$. Thus $P\bigcap Q$ is a prime ideal which implies $P\subseteq Q$ or $Q\subseteq P$ so $P$ or $Q$ is a prime ideal in $A(X)$. Without loss of generality we can assume that  $P=P^\frac{1}{3}_*$  is prime in $A(X)$.  Let  $a,b\in A(X)$ and $t=ab\in P_*$. Therefore, $a^\frac{1}{3}b^\frac{1}{3}=t^\frac{1}{3}\in P^\frac{1}{3}_*$ implies that $ b^\frac{1}{3}\in P^\frac{1}{3}_*$ or $a^\frac{1}{3} \in P^\frac{1}{3}_*$, i.e., $a\in P_*$ or $b\in P_*$ which completes the proof.
 \end{proof} 
 Clearly the above proposition will valid if instead of  $\frac{1}{3},\frac{1}{5}$ the fractions $\frac{1}{2k+1},\frac{1}{2n+1}$  are used, respectively where $k,n \in \mathbb{N}$.
 \begin{remark}
Similar to the proof of the above proposition one can easily seen that if the condition  {\normalfont(\ref{29ordybehesht})} holds for a subring $A(X)$  of $F(X)$, then for any two prime ideals $P$ and $Q$ in $A(X)$ we have $PQ=P\cap Q$, inparticular $P^n=P$, where $n\in \mathbb{N}$.
 \end{remark}
Let $X$ be an arbitrary topological space, the condition {\normalfont (\ref{29ordybehesht})} be true for  a subring  $A(X)$ of $F(X)$ and  $\mathbb{R}(X) \subseteq A(X)$. Then for every $f,g\in A(X)$ we have $|f-g|\in A(X)$ and consequently $f\wedge g =\frac{1}{2}(f-g-|f-g|)\in A(X)$. The following proposition is analogous to \cite[Theorem 1.6]{Gillman} and  \cite[Theorem 3.6]{Deb Ray}, and so we state it without proof.
\begin{lemma}\label{sarab}
Let    the condition {\normalfont(\ref{29ordybehesht})} hold for a subring $A(X)$ of $F(X)$ and  $\mathbb{R}(X) \subseteq A(X)$. Then $A(X)$  is a lattice ordered subring of $F(X)$ and if  $\phi : A(X) \to A(Y)$ is a ring homomorphism, then $\phi$ is a lattice  homomorphism.
\end{lemma}
Finally, we consider  the following property of a  subring $A(X)$ of $F(X)$. \\
   Let $Y$ be a topological space and  $\phi:X\to Y$  be  a continuous mapping.
 \begin{equation}\label{24ordybehesht}
\text{If}~ f\in A(Y),  \text{ then}~ f\circ \phi \in A(X).
  \end{equation}
  \begin{remark}
  Clearly, the condition {\normalfont (\ref{24ordybehesht})} is true for $C(X)$. By {\normalfont \cite[Theorem 2.4]{Deb Ray}},   the condition {\normalfont (\ref{24ordybehesht})} is true for $B_1(X)$.
  \end{remark}
  The following example shows that if the condition {\normalfont(\ref{29ordybehesht})} is true for a subring $A(X)$ of $F(X)$,  then  the condition {\normalfont(\ref{24ordybehesht})} need not be true for it.
\begin{example}
Let $X=Y=\mathbb{R}$ and $f$ be the sign function on $\mathbb{R}$. Then $\phi:\mathbb{R} \to \mathbb{R}$ defined by $\phi(x)=sin(x)$ is a  continuous function and $f\in C(X)_F$ but $f\circ \phi \notin C(X)_F$. 
\end{example}
 
\begin{theorem}\label{25KH}
Let    the conditions {\normalfont(\ref{29ordybehesht})} and {\normalfont(\ref{24ordybehesht})} hold for a subring  $A(X)$ of $F(X)$ and  $\mathbb{R}(X) \subseteq A(X)$. Then for a continuous mapping $\psi:X\to Y$ the mapping $A(\psi): A (Y) \to A (X)$ defined by
\begin{equation*}
A(\psi)(f)=f \circ \psi ~\text{for all }~ f\in A(Y),
\end{equation*}
is a lattice ordered ring homomorphism that preserves the unity element. 
\end{theorem}
\begin{proof}
It follows from Lemma \ref{sarab} since it is easily seen that $A(\psi)$ is a ring homomorphism.
 \end{proof}
 \begin{corollary} 
Let    the conditions {\normalfont(\ref{29ordybehesht})} and {\normalfont(\ref{24ordybehesht})} hold for a subring  $A(X)$ of $F(X)$ and  $\mathbb{R}(X) \subseteq A(X)$. If $\psi:X\to Y$ is a homeomorphism between topological spaces , then there exists a 
ring isomorphism between  $A (X)$ and  $A (Y)$.
\end{corollary}
 \begin{theorem}
Let    the conditions {\normalfont(\ref{29ordybehesht})} and {\normalfont(\ref{24ordybehesht})} hold for a subring  $A(X)$ of $F(X)$ and  $\mathbb{R}(X) \subseteq A(X)$. Then $A:X\to A(X), ~ \psi \to A(\psi)$ is the contravariant functor of the category of all topological spaces and the continuous mappings into the category of commutative   lattice ordered rings with  identity with  lattice ordered ring homomorphisms that preserve the unity element as its set of morphisms.
 \end{theorem}

\begin{proof}
If $\psi : X\to Y$ and  $\phi : Y\to Z$ are   continuous mappings, then
 \begin{equation*}
 A(\phi \circ \psi)(f)=f\circ (\phi \circ \psi)=(f \circ \phi )\circ \psi=A(\psi)(f \circ \phi)=A(\psi)(A(\phi)(f))=(A(\psi)\circ A(\phi))(f).
 \end{equation*}
Thus by Theorem \ref{25KH} we are done.
\end{proof}

 Let us recall that an ideal $J$ in a lattice-ordered ring $R$ is called 
convex if $f \in R, ~g \in J$, and   $0\leq f \leq g$  imply  $f\in J$ \cite{Gillman}. If $I$ is a convex ideal in $R$, then the quotient ring $R/I$ is a partially ordered ring, where for any $r\in R$, $r+I\geq 0$ if and only if there exists $s\in R$ such that $0\leq s$ and $r+I=s+I$  \cite[Theorem 5.2]{Gillman}. 
 An ideal $J$ in  $R$ is called  absolutely convex
 if $f \in R, ~g \in J$, and   $|f| \leq |g|$ imply  $f\in J$ \cite{Gillman}.\\
By Lemma \ref{sarab}, we can state the following definition.
\begin{definition}
Let    the condition {\normalfont(\ref{29ordybehesht})} hold for a subring $A(X)$ of $F(X)$ and  $\mathbb{R}(X) \subseteq A(X)$. Then  $A(X)$    is called $P$-convex if every prime ideal in $A(X)$ is  an absolutely convex ideal.
\end{definition}

\begin{example}
It is well known that the rings $C(X)$ and $F(X)$ are $P$-convex {\normalfont \cite{Gillman}}. By {\normalfont \cite[Proposition 4.8]{GGT}} and{\normalfont Example \ref{Ex}}, the subring $C(X)_F$ of $F(X)$ is $P$-convex.
\end{example}
The condition {\normalfont(\ref{29ordybehesht})} is true for the  lattice ordered ring $B_1(X)$ by Example \ref{Ex}. The next result shows that every prime idela in $B_1(X)$ is  absolutely convex.
\begin{theorem}
 $B_1(X)$ is  a $P$-convex subring of $F(X)$.
 \end{theorem}
 \begin{proof} 
 Let $P$ be an arbitrary prime ideal in $B_1(X)$ and $|f|\leq |g|$, where $g\in P$ and $ f\in B_1(X)$. Then $|g|\in P$ and  there exist two sequences  $\{f_n \}$ and  $\{g_n \}$ in $C(X)$ such that,  $\{f_n \}$ converges pointwise to $f$ and
 $\{g_n \}$ converges pointwise to $g$ on $X$. For any $n\in \mathbb{N}$, consider $h_n=|f_n|\vee| g _n|$ and $k_n : X\to \mathbb{R}$ defined by

 $$
k_n(x)= \left\{
\begin{array}{lc}
0						& x\in Z(h_n)\\
f^2_n(x) / h_n(x)  &  x\in X\setminus Z(h_n)
 \end{array}
\right.
$$
Since  $|f_n(x)|/ |h_n(x)|$ is bounded on $X\setminus Z(h_n)$, we observe that $k_n$ is  continuous. 
Clearly,  $\{h_n \}$ converges pointwise to $|g|$ and  $\{k_n \}$ converges pointwise to $k: X\to  \mathbb{R}$ defined by
 $$
k(x)= \left\{
\begin{array}{lc}
0						& x\in Z(g)\\
f^2(x) / |g(x)|  &  x\in X\setminus Z(g)
 \end{array}
\right.
$$
Clearly, $f^2=|g|k \in P$ and so $f\in P$ which completes the proof.
\end{proof} 

\begin{lemma}\label{30k}
Let $A(X)$ be a  $P$-convex subring of $F(X)$ and $I$ be an intersecton of a collection of prime ideals in $A(X)$. Then, $I$ is an absolutely convex ideal in $A(X)$.
\end{lemma}
\begin{proof}
Since $A(X)$ is a  $P$-convex subring of $F(X)$, every prime ideal in $A(X)$ is absolutely convex. It is well known that in any lattice ordered ring, an arbitrary intersection of absolutely convex ideals is 
 absolutely convex \cite[5B1]{Gillman}. Thus we are done.
\end{proof}
Let  $I$ be an ideal in a commutative ring $R$ with unity and $ \sqrt{I}$ denote the set of all elements of which some power belongs to $I$. Then $I\subseteq  \sqrt{I}$ and $ \sqrt{I}$ is the intersection of all prime ideals containing $I$ \cite[0.18]{Gillman}.
\begin{theorem}\label{siminar11}
Let the condition {\normalfont(\ref{mahmood})} be true  for a subring  $A(X)$ of $F(X)$ and $\mathbb{R}(X)\subseteq A(X)$. Then a $z$-ideal $I$ in $A(X)$ is an intersection of a collection of prime ideals in $A(X)$. 
\end{theorem}
\begin{proof}
If $f\in \sqrt{I}$, then $f^n \in I$ for some $n\in \mathbb{N}$. From  $Z(f)=Z(f^n)\in Z[I]$ and  Proposition \ref{zideal} it follows that  $f\in Z^{-1}[Z[I]]=I$. Thus $\sqrt{I} \subseteq  I$ which by the above comment completes the proof.
\end{proof}
\begin{remark}
Let $A(X)$ be a  $P$-convex subring of $F(X)$. Then every maximal ideal in $A(X)$ is an absolutely convex ideal. Moreover, let the condition {\normalfont(\ref{mahmood})} be true  for  $A(X)$. Then every  $z$-ideal in $A(X)$  is an absolutely convex ideal 
by {\normalfont Lemma \ref{30k}}  and  {\normalfont Theorem \ref{siminar11}}.
\end{remark}

\begin{theorem}\label{31kh}
Let $A(X)$ be a $P$-convex subring of $F(X)$, $I$ be an intersecton of a collection of prime ideals in $A(X)$, $f\in A(X)$ and the condition {\normalfont (\ref{mahmood})} be true for $A(X)$. Then $ I+f \geq 0 $ implies that there exists $ Z \in Z[I]$ such that  $0 \leq f  $ on $Z$.
\end{theorem}

\begin{proof}
By Lemma \ref{30k}, $I$ is an absolutely convex ideal in $A(X)$. Thus by \cite[Theorem 5.3(5)]{Gillman}, $ I+f =I+|f|$ since  $0\leq I+f \in A(X)/I$. Thus, $f-|f|\in I$ and so $Z=Z(f-|f|)\in Z[I]$ and $Z\neq \emptyset$ by the condition {\normalfont (\ref{mahmood})}. Clearly, $f_{|_{Z}}=|f|_{|_{Z}}\geq 0$.
\end{proof} 
In the following theorem, it is shown that if $I$ is a $z$-ideal in an especial subrings of $F(X)$, then the converse of Theorem \ref{31kh} holds.
\begin{theorem}\label{emam}
Let $A(X)$ be a $P$-convex subring of $F(X)$, $I$ be a $z$-ideal in $A(X)$, $f \in A(X)$ and  the condition {\normalfont (\ref{mahmood})} be true for $A(X)$. If there exists $Z\in Z[I]$ such that $f \geq 0$ on $Z$, then $0\leq I+f \in A(X)/I$.
\end{theorem}
\begin{proof}
  Since $ f = | f|$ on $Z$, we have $ Z \subseteq 
Z(f - | f|) $ and so $ Z(f - | f|) \in  Z[I]$.  By Proposition \ref{zideal}, $f - | f| \in I$, i.e.,
 $I+f = I+| f|$. Thus by \cite[Theorem 5.2]{Gillman},  $I+f \geq 0$ since $0\leq | f| $.
 \end{proof}
\begin{theorem}
  Let $A(X)$ be a $P$-convex subring of $F(X)$, the condition {\normalfont ( \ref{mahmood})} hold for $A(X)$, $I$ be a proper $z$-ideal in $A(X)$ and $f \in A(X)$. If there exists  $Z \in Z[I]$
such that $f_{|_{Z}}>0$, then  $I+f > 0$.
 \end{theorem}
\begin{proof}
 Since $ Z \bigcap Z(f) =\emptyset$,  $Z \in  Z[I]$ and by  the condition ( \ref{mahmood})  $\emptyset \notin Z[I]$, we have $Z(f)\notin Z[I]$. Since   $I$ is a $z$-ideal in $A(X)$ and by Proposition \ref{zideal}, we have $f\notin I$. Thus $I+f > 0$ since by  Theorem \ref{emam},  $ I+f \geq 0$.
 \end{proof}  
\begin{theorem}
  Suppose that  the condition {\normalfont ( \ref{mahmood})} is true for a  $P$-convex subring $A(X)$  of $F(X)$,   $M_i$ is a maximal ideal in $A(X)$ for  $1\leq i\leq n$  and $f \in A(X)$.  If $0<f+M_i$ for every $1\leq i\leq n$ , then there exists $Z\in Z[\bigcap_{i=1}^{i=n}M_i]$ such that $f$ is positive on $Z$.
 \end{theorem}
 \begin{proof}
 For any $i$, $0<f+M_i$ implies that $f\notin M_i$ and  so there exists $g_i \in M_i$ such that $Z(f)\bigcap Z(g_i)=\emptyset $ by Proposition \ref{KHorshidKhordad}.  By Theorem \ref{31kh},  there exists $Z_i\in Z[M_i]$ such that $f_{|_{Z_i}}\geq 0$. If $Z^{'}_i=Z_i \cap Z(g_i) $, then   $Z^{'}_i\in Z[M_i]$.  Clearly, $Z=\bigcup_{i=1}^{i=n}Z^{'}_i \in Z[\cap_{i=1}^{i=n}M_i]$ and $f$ is positive on $Z$. 
 \end{proof}
 %
 %

\end{document}